\documentclass[12pt]{article}

\usepackage[margin=0.95in]{geometry}
\usepackage[T1]{fontenc}
\usepackage{libertinus}
\usepackage{microtype}
\usepackage{amsmath,amssymb,amsthm}
\usepackage[hidelinks]{hyperref}
\usepackage[nameinlink,noabbrev]{cleveref}

\hypersetup{
  pdftitle={The sharp m-to-the-two-thirds constant for triangle independence and covering},
  pdfauthor={Author name omitted for review},
  pdfsubject={Extremal graph theory},
  pdfkeywords={triangle-independent edge set, triangle edge cover, Erdos--Gallai--Tuza problem}
}

\newtheorem{theorem}{Theorem}[section]
\newtheorem{lemma}[theorem]{Lemma}
\newtheorem{proposition}[theorem]{Proposition}
\newtheorem{corollary}[theorem]{Corollary}
\crefname{theorem}{theorem}{theorems}
\crefname{lemma}{lemma}{lemmas}
\crefname{proposition}{proposition}{propositions}
\crefname{corollary}{corollary}{corollaries}
\theoremstyle{remark}

\newtheorem{problem}[theorem]{Problem}

\title{Sharp asymptotics for triangle independence and covering numbers}
\author{Zhen Liu\footnote{Email: 1552580575@qq.com}, ~Qinghou Zeng\footnote{Research supported by National Key R\&D Program of China (Grant No. 2023YFA1010202) and National Natural Science Foundation of China (Grant No. 12371342). Email: zengqh@fzu.edu.cn (Corresponding
author)}\\
{\small Center for Discrete Mathematics, Fuzhou University, Fujian, 350003, China}}
\date{}

\begin{document}
\maketitle

\begin{abstract}
For a graph $G$, let $\alpha_1(G)$ be the maximum size of an edge set containing at most one edge from every triangle, and let $\tau_1(G)$ be the minimum size of an edge set meeting every triangle. Erdős, Gallai, and Tuza proved that $\alpha_1(G)+\tau_1(G)=\Omega(m^{2/3})$ for every $m$-edge graph and asked for the optimal asymptotic constant. We prove
\[
\lim_{m\to\infty} \min_{\substack{G\\ |E(G)|=m}} \frac{\alpha_1(G) + \tau_1(G)}{m^{2/3}} = \frac{3}{2},
\]
thereby establishing that the sharp constant is $3/2$ and solving the problem.
\end{abstract}


\section{Introduction}
Throughout the paper, all graphs are simple and undirected. Let $G$ be a graph with vertex set $V(G)$ and edge set $E(G)$; we write $e(G)=|E(G)|$ for the number of edges. A triangle is a clique on three vertices.

We will use the following triangle-related graph invariants.
\begin{itemize}
\item $\alpha_1(G)$  denotes the maximum size of an edge set that contains at most one edge from each triangle of $G$. In particular, $\alpha_1(G)$ is called the \emph{triangle-independence number} of $G$.
\item $\tau_1(G)$  denotes the minimum size of an edge set that contains at least one edge from each triangle of $G$. Equivalently, $\tau_1(G)$ is the minimum number of edges whose removal makes $G$ triangle-free, and is called the \emph{triangle edge cover number}.
\end{itemize}

Erd\H{o}s, Gallai, and Tuza introduced these parameters in \cite{EGT96}. They proved that
\[
\alpha_1(G)+\tau_1(G)\ge c\,e(G)^{2/3}
\]
for an absolute constant $c>0$, and a split-graph construction shows that the exponent $2/3$ is best possible. Their Problem~12, also recorded as Problem~48 in \cite{Tuza01}, asks for the optimal asymptotic constant.
\begin{problem}[Erd\H{o}s, Gallai and Tuza, \cite{EGT96}]
Prove that the function
\[
\displaystyle \min_{\substack{G\\ |E(G)|=m}}  \frac {\alpha_1(G) + \tau_1(G)}{m^{2/3}}
\]
tends to a constant as $m \to \infty$, and determine its value.
\end{problem}
The bounds recorded in \cite{Tuza01} are
\[
\liminf_{m\to\infty}\min_{\substack{G\\ |E(G)|=m}}
\frac{\alpha_1(G)+\tau_1(G)}{m^{2/3}}
\ge \frac{1}{\sqrt[3]{6}},
\qquad
\limsup_{m\to\infty}\min_{\substack{G\\ |E(G)|=m}}
\frac{\alpha_1(G)+\tau_1(G)}{m^{2/3}}
\le \sqrt[3]{4}.
\]
A separate line of work concerns the vertex-order inequality $\alpha_1(G)+\tau_1(G)\le \frac{|V(G)|^2}{4}.$ Puleo \cite{Puleo15} and Xu \cite{Xu17} obtained partial results for these vertex-order questions. Norin and Sun \cite{NorinSun16} proved the stronger inequality $\alpha_1(G)+\tau_B(G)\le \frac{|V(G)|^2}{4},$ where $\tau_B(G)$ is the minimum number of edges whose deletion makes $G$ bipartite. Since $\tau_1(G)\le\tau_B(G)$, their result implies the preceding vertex-order inequality. More recently, Bujt{\'a}s et al.~\cite{BujtasEtAl25} studied coverings of $E(G)$ by edges and triangles and established sharp bounds involving $\alpha_1(G)$.

Our main result is the following.

\begin{theorem}\label{thm:main}
As $m \to \infty$,
\[
\lim_{m\to\infty} \min_{\substack{G\\ |E(G)|=m}} \frac{\alpha_1(G) + \tau_1(G)}{m^{2/3}} = \frac{3}{2}.
\]
\end{theorem}
\paragraph{Notation.}
For a graph $G$, $\chi(G)$ denotes its chromatic number. If $A\subseteq V(G)$, then $G[A]$
denotes the subgraph of $G$ induced by $A$, and $G-A$ means
$G[V(G)\setminus A]$. A spanning subgraph of $G$ is a subgraph with
vertex set $V(G)$. We write $E_G(A)=E(G[A])$; if
$A,B\subseteq V(G)$ are disjoint, then $E_G(A,B)$ denotes the set of
edges of $G$ with one endpoint in $A$ and the other in $B$. We omit
the subscript $G$ when the ambient graph is clear. For a positive
integer $k$, set $\mathbb Z_k=\mathbb Z/k\mathbb Z$. For
vertex-disjoint graphs $G$ and $H$, $G\vee H$ is obtained from their
disjoint union by adding all edges between $V(G)$ and $V(H)$.

\section{Proof of main result}

\begin{proposition}\label{prop:lower}
Every graph $G$ with $m\ge1$ edges satisfies
\[
\alpha_1(G)+\tau_1(G)
\ge
\min_{k\ge1}\left\{\frac{m}{k}+\frac{(k-1)^2}{2}\right\},
\]
where the minimum is over positive integers $k$.
\end{proposition}

\begin{proof}
Let $C$ be a spanning subgraph of $G$ such that $E(C)$ is a minimum triangle edge cover; thus $|E(C)|=\tau_1(G)$.
Set $k=\chi(C)$ and let $R$ be the spanning subgraph with edge set $E(R)=E(G)\setminus E(C)$. Because $E(C)$ meets every triangle, $R$ is triangle‑free.

Fix a proper coloring $\varphi\colon V(C)\to\mathbb{Z}_k$. For $a\in\mathbb{Z}_k$ define
\[
R_a=\{\,uv\in E(R):\varphi(u)+\varphi(v)\equiv a\pmod{k}\,\}.
\]
The family $\{R_a\}$ partitions $E(R)$. We claim each $R_a$ is triangle‑independent.
If some triangle contained two edges $uv,uw\in R_a$, then its third edge $vw$ must lie in $E(C)$ (the two edges are not in $E(C)$ and $E(C)$ is a triangle cover). From the definition of $R_a$ we obtain
$\varphi(u)+\varphi(v)\equiv\varphi(u)+\varphi(w)\equiv a\pmod{k}$, hence $\varphi(v)=\varphi(w)$, contradicting the properness of $\varphi$ on the edge $vw\in E(C)$. Therefore $|R_a|\le\alpha_1(G)$ for every $a\in\mathbb{Z}_k$.
By the pigeonhole principle there exists $b\in\mathbb{Z}_k$ with
$|R_b|\ge|E(R)|/k=(m-\tau_1(G))/k$, yielding
\[
\alpha_1(G)\ge\frac{m-\tau_1(G)}{k}. \tag{1}
\]

To bound $\tau_1(G)$, note that $\varphi$ partitions $V(C)$ into $k$ colour classes. If two distinct classes contained no edge of $C$, merging them would give a proper $(k-1)$-colouring of $C$, contradicting $\chi(C)=k$. Hence every pair of colour classes contributes at least one edge to $C$, so
\[
\tau_1(G)=|E(C)|\ge\binom{k}{2}. \tag{2}
\]

Combining (1) and (2),
\[
\alpha_1(G)+\tau_1(G)\ge\frac{m-\tau_1(G)}{k}+\tau_1(G)
\ge\frac{m}{k}+\Bigl(1-\frac1k\Bigr)\binom{k}{2}
=\frac{m}{k}+\frac{(k-1)^2}{2}.
\]
This inequality holds for the chromatic number $k$ of any minimum triangle edge cover of $G$, from which the desired bound follows.
\end{proof}

\begin{corollary}\label{cor:explicit-lower}
For every graph $G$ with $e(G)=m$, $\alpha_1(G)+\tau_1(G)\ge \frac32\,m^{2/3}-2m^{1/3}.$
\end{corollary}

\begin{proof}
By Proposition~\ref{prop:lower} there exists a positive integer $k$ such that
\[
\alpha_1(G)+\tau_1(G)\ge \frac{m}{k}+\frac{(k-1)^2}{2}.
\]

\noindent\textit{Case 1:} $k\le 2m^{1/3}$. Then
\[
\frac{m}{k}+\frac{(k-1)^2}{2}
= \frac{m}{k}+\frac{k^2}{2}-k+\frac12
\ge \frac{m}{k}+\frac{k^2}{2}-2m^{1/3}.
\]
By the arithmetic--geometric mean inequality applied to  $\frac{m}{2k},\frac{m}{2k},\frac{k^2}{2}$ yields
\[
\frac{m}{k}+\frac{k^2}{2}
= \frac{m}{2k}+\frac{m}{2k}+\frac{k^2}{2}
\ge 3\sqrt[3]{\frac{m}{2k}\cdot\frac{m}{2k}\cdot\frac{k^2}{2}}
= \frac32\,m^{2/3}.
\]
Hence $\alpha_1(G)+\tau_1(G)\ge \frac32m^{2/3}-2m^{1/3}$.

\noindent\textit{Case 2:} $k>2m^{1/3}$. Then
\[
\alpha_1(G)+\tau_1(G)\ge \frac{(k-1)^2}{2}
> \frac{(2m^{1/3}-1)^2}{2}
= 2m^{2/3}-2m^{1/3}+\frac12
> \frac32\,m^{2/3}-2m^{1/3}.
\]
In both cases the desired inequality holds.
\end{proof}

For integers $q,r\ge 1$, let $H=K_q\vee\overline{K_r}$; this split graph consists of a $q$-clique $Q$, an independent set $I$ of size $r$, and all edges between $Q$ and $I$.

\begin{lemma}\label{lem:split}
For all $q,r\ge 1$, $\alpha_1(H)\le r+\Bigl\lfloor\frac{q}{2}\Bigr\rfloor$ and $\tau_1(H)\le\binom{q}{2}.$
\end{lemma}

\begin{proof}
Let $F$ be a triangle‑independent edge set in $H$.
We first bound $|F\cap E(Q)|$. If two edges of $F$ inside $Q$ shared a vertex, then together with the edge joining their other endpoints (which exists because $Q$ is a clique) they would form a triangle containing two edges of $F$, contradicting the triangle‑independence of $F$. Thus $F\cap E(Q)$ is a matching, and $|F\cap E(Q)|\le\lfloor q/2\rfloor$.
Now consider edges incident to the independent set $I$. Suppose some $v\in I$ is incident to two edges $vu,vw\in F$ with $u,w\in Q$. Since $Q$ is a clique, $uw\in E(H)$. Then $\{u,v,w\}$ is a triangle containing two edges of $F$, again a contradiction. Hence every vertex of $I$ is incident to at most one edge of $F$, so $|F\cap E(Q,I)|\le r$. As $I$ is independent, $F$ consists only of edges inside $Q$ and edges between $Q$ and $I$, giving
\[
|F| = |F\cap E(Q)| + |F\cap E(Q,I)| \le \Bigl\lfloor\frac{q}{2}\Bigr\rfloor + r.
\]
This proves the bound on $\alpha_1(H)$.

For $\tau_1(H)$, note that removing all edges inside $Q$ leaves the complete bipartite graph between $Q$ and $I$, which is bipartite and therefore triangle‑free. Hence the set of $\binom{q}{2}$ edges inside $Q$ is a triangle edge cover, yielding $\tau_1(H)\le\binom{q}{2}$.
\end{proof}

\begin{proof}[Proof of Theorem~\ref{thm:main}]
The lower bound $\alpha_1(G)+\tau_1(G)\ge \frac32\,m^{2/3}-2m^{1/3}$ follows from Corollary~\ref{cor:explicit-lower}. For the upper bound we construct, for all large $m$, an $m$-edge graph $G_m$ achieving $\frac32\,m^{2/3}+O(m^{1/3})$.

Set $q=\lfloor m^{1/3}\rfloor$; then $q^3\le m<(q+1)^3$, so $m=q^3+O(q^2)$. Let $r=\Bigl\lfloor\frac{m-\binom{q}{2}}{q}\Bigr\rfloor$, $s=m-\binom{q}{2}-qr.$ Since \(r = \bigl\lfloor (m-\binom{q}{2})/q \bigr\rfloor\), the division algorithm yields \(m-\binom{q}{2} = qr + s\) with \(0 \le s < q\). Take $H=K_q\vee\overline{K_r}$ and let $G_m$ be the disjoint union of $H$ and a matching of size $s$. Then $e(G_m)=m$. By Lemma~\ref{lem:split},
\[
\alpha_1(H)\le r+\Bigl\lfloor\frac{q}{2}\Bigr\rfloor,\qquad \tau_1(H)\le\binom{q}{2},
\]
and for the matching component $\alpha_1=s$, $\tau_1=0$. Since both invariants are additive over components,
\begin{align*}
\alpha_1(G_m)+\tau_1(G_m)
&\le r+\frac{q}{2}+s+\binom{q}{2}+O(1) = r+s+\frac{q^2}{2}+O(1).
\end{align*}
Because $r = \left\lfloor \big(m - \binom{q}{2}\big)/q \right\rfloor$, $q^2 = m^{2/3} + O(m^{1/3})$, and $0 \le s < q$, we obtain
\[
\alpha_1(G_m) + \tau_1(G_m) \le \frac{3}{2}m^{2/3} + O(m^{1/3}),
\]
which, together with the lower bound, yields
$$
\lim_{m\to\infty} \min_{\substack{G\\ |E(G)|=m}} \frac{\alpha_1(G) + \tau_1(G)}{m^{2/3}} = \frac{3}{2},
$$
completing the proof.
\end{proof}

\section*{Declaration on the Use of Generative AI}
The authors used ChatGPT 5.6 Pro to assist in discussing proof strategies, checking proofs, and
improving exposition.

\end{document}